\documentclass{amsart}

\newtheorem{theorem}{Theorem}[section]
\newtheorem{lemma}[theorem]{Lemma}

\theoremstyle{definition}

\theoremstyle{remark}

\begin{document}

\title{About some identities for Bessel polynomials}
\author[Leveque]{Olivier L\'{e}v\^{e}que}
\address{L.T.H.I., E.P.F.L., Lausanne, Switzerland}
\email{olivier.leveque@epfl.ch}
\author[C. Vignat]{Christophe Vignat}
\address{Department of Mathematics,
Tulane University, New Orleans, LA 70118 and 
L.S.S. Supelec, Universite d'Orsay, France}
\email{vignat@tulane.edu}
\maketitle

\begin{abstract}
In a recent paper, Yu. A. Brychkov derived a series of identities for multiples sums of special functions, using generating functions. Among these identities, a particularly interesting one involves multiples sums of Bessel $I_{\nu}$ functions with half-integer indices. We derive here some equivalent identities that involve different kinds of Bessel polynomials, sing a probabilistic approach based on the properties of the Generalized Inverse Gaussian probability density.
\end{abstract}

\section{Introduction}

In a recent paper \cite{Brychkov}, using generating functions,
Yu. A. Brychkov derived the following identity
\begin{gather}
\sum_{k_{1}+\dots+k_{m}=n}\prod_{i=1}^{m}\frac{1}{k_{i}!}\left[I_{-k_{i}-\frac{1}{2}}\left(z\right)-I_{k_{i}+\frac{1}{2}}\left(z\right)\right]=\left(-1\right)^{n}\frac{m^{\frac{1}{2}}}{n!}\pi^{\frac{1-m}{2}}\left(\frac{z}{2}\right)^{\frac{1-m}{2}-n}\label{eq:Brychkov}\\
\times\sum_{k=0}^{n}\binom{n}{k}\left(\frac{m-1}{2}\right)_{n-k}\left(-\frac{mz}{2}\right)^{k}\left[I_{-k-\frac{1}{2}}\left(mz\right)-I_{k+\frac{1}{2}}\left(mz\right)\right].\nonumber 
\end{gather}
where $I_{k}\left(z\right)$ is the modified Bessel function of the
first kind. Our aim here is to show that this formula is equivalent to several simple identities involving various kinds of Bessel polynomials, and to provide some
extended versions of it. Our approach involves a probabilistic interpretation of these identities which does not require the computation of any generating function.

Let us first remark that, from \cite[p. 675]{Brychkovbook},
\[
I_{-k-\frac{1}{2}}\left(z\right)-I_{k+\frac{1}{2}}\left(z\right)=\frac{2}{\pi}\left(-1\right)^{k}K_{k+\frac{1}{2}}\left(z\right)
\]
where $K_{k}\left(z\right)$ is the modified Bessel function of the
third kind or Macdonald function. For $k$ integer, this function
is related to the Bessel polynomial $q_{k}\left(z\right)$ of degree $k$ as follows
\begin{equation}
\exp\left(-z\right)q_{k}\left(z\right)=\frac{2^{\frac{1}{2}-k}}{\Gamma\left(k+\frac{1}{2}\right)}z^{k+\frac{1}{2}}K_{k+\frac{1}{2}}\left(z\right);\,\, z\ge0\label{eq:Bessel}
\end{equation}
where the Bessel polynomial $q_{k}\left(z\right)$ is defined by 
\[
q_{k}\left(z\right)=\sum_{l=0}^{k}\frac{\binom{k}{l}}{\binom{2k}{l}}\frac{\left(2z\right)^{l}}{l!}.
\]
First examples of these Bessel polynomials are
\[
q_{0}\left(z\right)=1,\,\, q_{1}\left(z\right)=1+z,\,\, q_{2}\left(z\right)=1+z+\frac{z^{2}}{3}.
\]
These polynomials satisfy the normalization constraint
\[
q_{n}\left(0\right)=1.
\]
Replacing the Bessel functions $I_{k}\left(z\right)$ by their expression
in terms of Bessel polynomials yields, after some elementary algebra, the equivalent
and more compact version of identity (\ref{eq:Brychkov})
\begin{equation}
\sum_{k_{1}+\dots+k_{m}=n}\prod_{i=1}^{m}\binom{2k_{i}}{k_{i}}q_{k_{i}}\left(z\right)=\sum_{k=0}^{n}\binom{2k}{k}2^{2n-2k}\frac{\left(\frac{m-1}{2}\right)_{n-k}}{\left(n-k\right)!}q_{k}\left(mz\right).\label{eq:convolutionform}
\end{equation}
Moreover, replacing the Bessel functions $I_{k}\left(z\right)$ by
their expression in terms of Bessel $K_{k}\left(z\right)$ functions
yields the equivalent identity
\[
\sum_{k_{1}+\dots+k_{m}=n}\prod_{i=1}^{m}\frac{z^{k_{i}+\frac{1}{2}}K_{k_{i}+\frac{1}{2}}(z)}{k_{i}!}=\left(\frac{\sqrt{\pi}}{2}\right)^{m-1}\sum_{k=0}^{n}\frac{\left(\frac{m-1}{2}\right)_{n-k}}{\left(n-k\right)!}\frac{\left(\frac{mz}{2}\right)^{k+\frac{1}{2}}K_{k+\frac{1}{2}}\left(mz\right)}{k!}.
\]

\section{A probabilistic approach to identity (\ref{eq:convolutionform})}

We show here that identity (\ref{eq:convolutionform}) can be interpreted
in a probabilistic setting; this interpretation relies on
properties of the generalized inverse Gaussian distribution, defined as \cite{Jorgensen}
\begin{equation}
f\left(x;\psi,\chi,\lambda\right)=\frac{\left(\frac{\psi}{\chi}\right)^{\frac{\lambda}{2}}}{2K_{\lambda}\left(\sqrt{\psi\chi}\right)}x^{\lambda-1}\exp\left(-\frac{\chi}{2} x^{-1}-\frac{\psi}{2} x\right),\,\, x,\psi,\chi>0,\,\,\lambda>-1.\label{eq:GIG}
\end{equation}

In order to explicit this probabilistic interpretation, we need the following preliminary results. 
This first lemma can be found in \cite{Jorgensen}. 
\begin{lemma}
\cite{Jorgensen}\label{lem:stability} Let $X_{-\frac{1}{2},z_{1}}$
and $X_{-\frac{1}{2},z_{2}}$ two independent random variables with
generalized inverse distribution as in (\ref{eq:GIG}) with parameters
$\psi=1,\,\,\chi=z^{2}$ and $lambda=-\frac{1}{2},$ then
\begin{equation}
X_{-\frac{1}{2},z_{1}}+X_{-\frac{1}{2},z_{2}}\sim X_{-\frac{1}{2},z_{1}+z_{2}}.\label{eq:stability}
\end{equation} where the sign $\sim$ denotes identity in distribution.
Moreover, 
\begin{equation}
X_{-\frac{1}{2},z}+X_{\frac{1}{2},0}\sim X_{\frac{1}{2},z}.\label{eq:decomposition}
\end{equation}

\end{lemma}
From these two results, we can deduce the following identity, that
can also be found in \cite{Jorgensen}
\begin{lemma}
\label{lem:samedistribution}Let $X_{\frac{1}{2},z_{1}}$ and $X_{\frac{1}{2},z_{2}}$
two independent random variables with generalized inverse Gaussian
distribution as in (\ref{eq:GIG}) with parameters $\psi=1,\,\,\chi=z^{2},\,\,\lambda=\frac{1}{2},$
then the identity in distribution
\begin{equation}
X_{\frac{1}{2},z_{1}}+X_{\frac{1}{2},z_{2}}\sim X_{\frac{1}{2},0}+X_{\frac{1}{2},z_{1}+z_{2}}\label{eq:convolution1/2}
\end{equation}
holds, where $X_{\frac{1}{2},z_{1}+z_{2}}$ is independent of $X_{\frac{1}{2},0}$.
\end{lemma}
\begin{proof}
We remark from (\ref{eq:decomposition}) that
\[
X_{\frac{1}{2},z_{1}}+X_{\frac{1}{2},z_{2}}\sim\left(X_{-\frac{1}{2},z_{1}}+X_{\frac{1}{2},0}\right)+\left(X_{-\frac{1}{2},z_{2}}+\tilde{X}_{\frac{1}{2},0}\right)
\]
where the random variables on the right-hand side are mutually independent.
By the stability property (\ref{eq:stability}), we obtain
\begin{eqnarray*}
X_{\frac{1}{2},z_{1}}+X_{\frac{1}{2},z_{2}} & \sim & X_{-\frac{1}{2},z_{1}+z_{2}}+X_{\frac{1}{2},0}+\tilde{X}_{\frac{1}{2},0}\\
 & \sim & X_{\frac{1}{2},z_{1}+z_{2}}+X_{\frac{1}{2},0}
\end{eqnarray*}
where we have used again the property (\ref{eq:decomposition}), hence
the result.
\end{proof}
We note that the probability density of $X_{\frac{1}{2},0}$ is chi-square density with one degree of freedom, or equivalently a Gamma density with scale parameter $2$ and shape parameter $1/2$: 
\[
f_{X_{\frac{1}{2},0}}(x)=\frac{1}{2 \sqrt{\pi}} \left(\frac{x}{2}\right)^{-\frac{1}{2}} \exp\left(-\frac{x}{2}\right).
\]

The link between the Generalized Inverse Gaussian random variables
and the Bessel polynomials is characterized as follows.

\begin{lemma}
Given $\nu \in \mathbb{R}$, the $\nu-$th moment  of $X_{\frac{1}{2},z}$ is 
\begin{equation}
\mathbb{E}X_{\frac{1}{2},z}^{\nu}=\sqrt{\frac{2}{\pi}}\exp\left(z\right)z^{\nu+\frac{1}{2}}K_{\nu+\frac{1}{2}}\left(z\right).\label{eq:moment-2}
\end{equation}

When $\nu=n$ is an integer, this formula simplifies to
\begin{equation}
\mathbb{E}X_{\frac{1}{2},z}^{n}=\frac{1}{2^{n}}\frac{\left(2n\right)!}{n!}q_{n}\left(z\right)\label{eq:moment}
\end{equation}
where $q_{n}\left(z\right)$ is the Bessel polynomial of degree $n.$
\end{lemma}

\begin{proof}
The $\nu-$th moment is easily computed as the integral
\[
\int_{0}^{+\infty}x^{\nu}f\left(x;1,z^{2},\frac{1}{2}\right)dx=\sqrt{\frac{2}{\pi}}\exp\left(z\right)z^{\nu+\frac{1}{2}}K_{\nu+\frac{1}{2}}\left(z\right)
\]
where $K_{k}$ is the Bessel function of the second kind of order
$k.$ In the case where $\nu=n$ is an integer, using the expression
(\ref{eq:Bessel}) of the Bessel function $K_{n+\frac{1}{2}}$ in
terms of the Bessel polynomial $q_{n}$, we obtain
\begin{eqnarray*}
\mathbb{E}X_{\frac{1}{2},z}^{n} & = & \sqrt{\frac{2}{\pi}}\exp\left(z\right)z^{n+\frac{1}{2}}K_{n+\frac{1}{2}}\left(z\right)=\sqrt{\frac{2}{\pi}}\exp\left(z\right)\Gamma\left(n+\frac{1}{2}\right)2^{n-\frac{1}{2}}\exp\left(-z\right)q_{n}\left(z\right)\\
 & = & \frac{2^{n}}{\sqrt{\pi}}\Gamma\left(n+\frac{1}{2}\right)q_{n}\left(z\right)
\end{eqnarray*}
with $\frac{\Gamma\left(n+1/2\right)}{\Gamma\left(1/2\right)}=\frac{1}{2^{2n}}\frac{\left(2n\right)!}{n!},$
hence the result.
\end{proof}

We now have the necessary tools to prove the following extension of
identity (\ref{eq:convolutionform}).
\begin{theorem}
If $\left\{ z_{i},\,1\le i\le m;\, m\ge2\right\} $ are complex numbers
then, with $z=\sum_{i=1}^{m}z_{i},$
\begin{equation}
\sum_{k_{1}+\dots+k_{m}=n}\prod_{i=1}^{m}\binom{2k_{i}}{k_{i}}q_{k_{i}}\left(z_{i}\right)=\sum_{k=0}^{n}\binom{2k}{k}\frac{2^{2n-2k}\left(\frac{m-1}{2}\right)_{n-k}}{\left(n-k\right)!}q_{k}\left(z\right).\label{eq:general Bessel}
\end{equation}
The special case where all $z_{i}$ are equal reads Brychkov's identity
(\ref{eq:convolutionform}).
\end{theorem}
\begin{proof}
We consider first the case $m=2.$ As a consequence of Lemma \ref{lem:samedistribution},
the moments of $X_{\frac{1}{2},z_{1}}+X_{\frac{1}{2},z_{2}}$ are
the same as the moments of $X_{\frac{1}{2},0}+X_{\frac{1}{2},z_{1}+z_{2}}.$
Using the binomial expansion, the moment of order $n$ of $X_{\frac{1}{2},z_{1}}+X_{\frac{1}{2},z_{2}}$
yields 
\begin{eqnarray*}
\mathbb{E}\left(X_{\frac{1}{2},z_{1}}+X_{\frac{1}{2},z_{2}}\right)^{n} & = & \sum_{k=0}^{n}\binom{n}{k}\mathbb{E}\left(X_{\frac{1}{2},z_{1}}\right)^{k}\mathbb{E}\left(X_{\frac{1}{2},z_{2}}\right)^{n-k}\\
 & = & \sum_{k=0}^{n}\binom{n}{k}\frac{1}{2^{k}}\frac{\left(2k\right)!}{k!}q_{k}\left(z_{1}\right)\frac{1}{2^{n-k}}\frac{\left(2n-2k\right)!}{\left(n-k\right)!}q_{n-k}\left(z_{2}\right)\\
 & = & \frac{n!}{2^{n}}\sum_{k=0}^{n}\binom{2k}{k}\binom{2n-2k}{n-k}q_{k}\left(z_{1}\right)q_{n-k}\left(z_{2}\right).
\end{eqnarray*}
The same approach applied to $X_{\frac{1}{2},0}+X_{\frac{1}{2},z_{1}+z_{2}}$
gives
\begin{eqnarray*}
\mathbb{E}\left(X_{\frac{1}{2},0}+X_{\frac{1}{2},z_{1}+z_{2}}\right)^{n} & = & \sum_{k=0}^{n}\binom{n}{k}2^{k}\frac{\Gamma\left(k+\frac{1}{2}\right)}{\Gamma\left(\frac{1}{2}\right)}\frac{1}{2^{n-k}}\frac{\left(2n-2k\right)!}{\left(n-k\right)!}q_{n-k}\left(z_{1}+z_{2}\right)\\
 & = & \frac{1}{2^{n}}\sum_{k=0}^{n}\binom{n}{k}2^{2k}\frac{1}{2^{2k}}\frac{\left(2k\right)!}{k!}\frac{\left(2n-2k\right)!}{\left(n-k\right)!}q_{n-k}\left(z_{1}+z_{2}\right)\\
 & = & \frac{n!}{2^{n}}\sum_{k=0}^{n}\binom{2k}{k}\binom{2n-2k}{n-k}q_{k}\left(z_{1}+z_{2}\right),
\end{eqnarray*}
which yields the result. The extension to the case of an arbitrary
integer value $m>2$ is left to the reader, using the following elementary
extension of identity (\ref{eq:convolution1/2}):
\begin{lemma}
If $\left\{ z_{i},\,1\le i\le m;\, m\ge2\right\} $ are real positive
numbers and $X_{i}$ are independent random variables, then the following
identity in distribution holds
\[
\sum_{i=1}^{m}X_{\frac{1}{2},z_{i}}\sim X_{\frac{1}{2},\sum_{i=1}^{m}z_{i}}+X_{\frac{m-1}{2},0}.
\]
\end{lemma}
\end{proof}
We remark that $X_{\frac{m-1}{2},0}$ is distributed as a Gamma random
variable with scale parameter $2$ and shape parameter $\frac{m-1}{2}$ (or equivalently a chi random variable with $m-1$ degrees of freedom). 

\section{links to Laguerre polynomials}
We note that the Laguerre polynomials $L_{n}^{\left(\mu\right)}\left(z\right)$
are related to the Bessel polynomials $q_{n}\left(z\right)$ as
\[
q_{n}\left(z\right)=\frac{\left(-1\right)^{n}}{\binom{2n}{n}}L_{n}^{\left(-2n-1\right)}\left(2z\right);
\]
as a consequence, an equivalent statement of identity (\ref{eq:general Bessel})
in terms of Laguerre polynomials reads, with $z=\sum_{i=1}^{m} z_{i},$
\[
\sum_{k_{1}+\dots+k_{m}=n}\prod_{i=1}^{m}L_{k_{i}}^{\left(-2k_{i}-1\right)}\left(z_{i}\right)=\sum_{k=0}^{n}2^{2n-2k}\frac{\left(\frac{m-1}{2}\right)_{n-k}}{\left(n-k\right)!}\left(-4\right)^{n-k}L_{k}^{\left(-2k-1\right)}\left(z\right).
\]
The special case $m=2$ and $z_{1}=-z_{2}=z$ of this identity can
be found in \cite[4.4.2.9]{Prudnikov} as
\[
\sum_{k=0}^{n}L_{k}^{\left(-2k-1\right)}\left(z\right)L_{n-k}^{\left(-2n+2k-1\right)}\left(-z\right)=\left(-4\right)^{n}.
\]

\section{links to other Bessel polynomials}

In this section, we show that two other families of Bessel polynomials
can be interpreted as moments in the same way as identity (\ref{eq:moment}).
We deduce from these representations some identities equivalent to
(\ref{eq:general Bessel}) but that have a more simple form.

\subsection{Bessel $\theta_{n}$ polynomials}

In his textbook \cite{Grosswald}, Grosswald considers the following
Bessel polynomials, called \textit{reverse Bessel polynomials}
\begin{equation}
\theta_{n}\left(z\right)=\frac{\left(2n\right)!}{n!}2^{-n}q_{n}\left(z\right).\label{eq:theta_q}
\end{equation}
For example,
\[
\theta_{0}\left(z\right)=1,\,\,\theta_{1}\left(z\right)=1+z,\,\,\theta_{2}\left(z\right)=3+3z+z^{2}.
\]
From identity (\ref{eq:moment}), we deduce
\[
\theta_{n}\left(z\right)=\mathbb{E}X_{\frac{1}{2},z}^{n}.
\]
Moreover, the identity (\ref{eq:general Bessel}) reads in terms of
these polynomials
\[
\sum_{k_{1}+\dots+k_{m}=n}\prod_{i=1}^{m}\frac{\theta_{k_{i}}\left(z_{i}\right)}{k_{i}!}=\sum_{k=0}^{n}\frac{2^{n-k}\left(\frac{m-1}{2}\right)_{n-k}}{\left(n-k\right)!}\frac{\theta_{k}\left(z\right)}{k!}.
\]
with $z=\sum_{i=1}^{m} z_{i}.$
The case $m=2$ appears in \cite[eqn. (5.4)]{Carlitz}.

\subsection{Bessel $f_{n}$ polynomials}

The third family of Bessel polynomials $f_{n}$ is defined by L. Carlitz
\cite{Carlitz} as
\begin{equation}
f_{n}\left(z\right)=z\theta_{n-1}\left(z\right);\,\, n\ge1,\label{eq:f_theta}
\end{equation}
and $f_{0}\left(z\right)=1.$ First examples are
\[
f_{1}\left(z\right)=z;\,\, f_{2}\left(z\right)=z+z^{2};\,\, f_{3}\left(z\right)=3z+3z^{2}+z^{3}.
\]
Note that as in the case of the reverse Bessel polynomials, these
polynomials have their highest degree coefficient equal to $1.$ 

Since it can be easily checked from (\ref{eq:GIG}) that the Generalized Inverse Gaussian density satisfies the functional equation
\[
xf\left(x;1,z^{2},-\frac{1}{2}\right)=zf\left(x;1,z^{2},+\frac{1}{2}\right),
\]
we deduce
\[
z\mathbb{E}X_{\frac{1}{2},z}^{n-1}=\mathbb{E}X_{-\frac{1}{2},z}^{n}
\]
so that
\[
f_{n}\left(z\right)=\mathbb{E}X_{-\frac{1}{2},z}^{n},\,\,\forall n\ge1.
\]
 Since moreover $f_{0}\left(z\right)=1,$ this representation holds
in fact $\forall n\ge0.$ By the stability property (\ref{eq:stability}),
we deduce that the polynomials $f_{n}$ satisfy the multinomial property
\begin{equation}
\sum_{k_{1}+\dots+k_{m}=n}\prod_{i=1}^{m}\frac{f_{k_{i}}\left(z_{i}\right)}{k_{i}!}=\frac{f_{n}\left(z\right)}{n!}
\label{eq:f_multi}
\end{equation}
with $z=\sum_{i=1}^{m} z_{i}.$
The case $m=2$ is given in \cite[eqn. (2.7)]{Carlitz}. 

Moreover, since, by
(\ref{eq:f_theta}), (\ref{eq:theta_q}) and (\ref{eq:Bessel}), the
polynomials $f_{n}$ are related to the Bessel functions $K_{n}$
as
\[
f_{n}\left(z\right)=\sqrt{\frac{2}{\pi}}z^{n+\frac{1}{2}}K_{n-\frac{1}{2}}\left(z\right),
\]
we deduce an equivalent version of (\ref{eq:f_multi})
\[
\sum_{k_{1}+\dots+k_{m}=n}\prod_{i=1}^{m}\frac{z_{i}^{k_{i}+\frac{1}{2}}}{k_{i}!}K_{k_{i}-\frac{1}{2}}\left(z_{i}\right)=\left(\frac{2}{\pi}\right)^{\frac{1-m}{2}}\frac{\left(z\right)^{n+\frac{1}{2}}}{n!}K_{n-\frac{1}{2}}\left(z\right)
\]
with $z=\sum_{i=1}^{m} z_{i},$
which can be found in \cite[5.18.1.3]{Brychkovbook}.

\section{Conclusion}
We have shown that an identity introduced by Brychkov using generating
functions can be interpreted as a multiplication identity for several types of Bessel polynomials.
Moreover, we have given a probabilistic background to this identity
and exhibited its relationship with the Generalized Inverse Gaussian density. 
Other tools, such that generating functions, can certainly replace this probabilistic approach; however, we found that it is particularly convenient in this context.

As a final illustration of this efficiency of this tool, we derive a quick proof of a famous Tur\'{a}n-type inequality for Bessel $K_{\nu}$ functions: in \cite{Ismail}, Ismail and Muldoon proved that the function
\[
\nu \mapsto \frac{K_{\nu+b}\left(z \right)}{K_{\nu}\left(z \right)}
\]
is increasing, implying that the function
\[
\nu \mapsto K_{\nu}\left(z \right)
\]
is log-convex. As a consequence, 
the following Tur\'{a}n-type inequality - as named
by Karlin and Szeg\"{o} - for Bessel functions holds: for $\frac{x}{p}$ and
$\frac{y}{q}>-\frac{1}{2},$ and with $\frac{1}{p}+\frac{1}{q}=1,$
\[
K_{\frac{x}{p}+\frac{y}{q}}\left(z\right) \le K_{x}\left(z\right)^{\frac{1}{p}}K_{y}\left(z\right)^{\frac{1}{q}}.
\]
The case $p=q=2$ of this inequality is also derived in \cite{Laforgia} from the following generalization of the
Cauchy-Schwarz inequality
\[
\int_{a}^{b} g(t)f^{m}(t)dt\,\, \times \int_{a}^{b} g(t)f^{n}(t)dt \ge \left(\int_{a}^{b} g(t)f^{\frac{m+n}{2}}(t)dt\right)^{2}
\]

The probabilistic approach based on the moment representation
(\ref{eq:moment-2}) allows to derive a straightforward proof
of this result that does not require such a refinement:
choosing a random variable $X_{\frac{1}{2},z}$
and applying the standard H\"{o}lder inequality
\[
\mathbb{E}\left(Z_{1}Z_{2}\right) \le \left[ \mathbb{E}Z_{1}^{p}\right]^{\frac{1}{p}}  \left[ \mathbb{E}Z_{2}^{q}\right]^{\frac{1}{q}}
\]
with $\frac{1}{p}+\frac{1}{q}=1$ and $Z_{1}=X_{\frac{1}{2},z}^{\nu/2-1/4}$ and $Z_{2}=X_{\frac{1}{2},z}^{\mu/2-1/4},$
we deduce the result.

%
%

\label{lastpage}

\end{document}